\documentclass[12pt]{amsart}
\usepackage[utf8]{inputenc}
\usepackage{amsmath}
\usepackage{amssymb}
\usepackage{amsfonts}
\usepackage{amsthm}
\usepackage{mathrsfs} 
\usepackage{bm}
\usepackage{bbm}
\usepackage{tikz}
\usepackage{centernot}
\usepackage{hyperref}
\usepackage{enumitem}
\usepackage[ruled,vlined]{algorithm2e}
\usepackage[margin=1.2in]{geometry}

\newcommand{\E}{\Bbb{E}}
\newcommand{\Z}{\Bbb{Z}}
\newcommand{\R}{\Bbb{R}}

\newcommand{\F}{\Bbb{F}}
\newcommand{\T}{\Bbb{T}}

\renewcommand{\AA}{\mathcal{A}}

\newcommand{\ep}{\epsilon}

\hypersetup{
    colorlinks=true,
    linkcolor=blue,
    filecolor=magenta,
    urlcolor=blue,
    citecolor=blue,
    backref=true
}

\makeatletter
\newtheorem*{rep@theorem}{\rep@title}
\newcommand{\newreptheorem}[2]{%
\newenvironment{rep#1}[1]{%
 \def\rep@title{#2 \ref{##1}}%
 \begin{rep@theorem}}%
 {\end{rep@theorem}}}
\makeatother

\numberwithin{equation}{section}

\newtheorem{thm}{Theorem}
\newreptheorem{thm}{Theorem}

\newtheorem{result}{Result}[section]
\newtheorem{lem}[result]{Lemma}
\newtheorem{prp}[result]{Proposition}

\theoremstyle{definition}
\newtheorem{rmk}[result]{Remark}
\newtheorem{defn}{Definition}

\newtheorem*{ack}{Acknowledgements}

\theoremstyle{remark}

\newcommand{\hide}[1]{}
\newcommand{\edit}[1]{}

\newcommand{\rough}[1]{}
\definecolor{darkgreen}{RGB}{75,150,75}
\newcommand{\review}[1]{}

\newcommand{\zh}[1]{\textcolor{blue}{zh: #1}}

\newcommand{\hides}[1]{}
\newcommand{\pub}[1]{}

\title{New lower bounds for $r_3(N)$}
 
\author{Zach Hunter}
\address[Hunter]{Department of Mathematics, ETH, Z\"urich, Switzerland.}
\email{zach.hunter@math.ethz.ch}

\date{\today}

\begin{document}

\maketitle
\begin{abstract}
    We develop recent ideas of Elsholtz, Proske, and Sauermann to construct denser subsets of $\{1,\dots,N\}$ that lack arithmetic progressions of length $3$. This gives the first quasipolynomial improvement since the original construction of Behrend. 
\end{abstract}
\section{Introduction}

We say an additive set $S$ is \textit{$3$-AP-free}, if it lacks non-trivial arithmetic progressions of length three (meaning that all solutions to $2y=x+z$ are of the form $x=y=z$). Let $r_3(N)$ denote the maximum cardinality of a $3$-AP-free subset of $\{1,\dots,N\}$.

For many years, the best-known lower bound has been\footnote{See Section~\ref{prelim} for our asymptotic notation.}
\begin{equation}\label{roth bound}
    r_3(N)\gg \log^C (N) N2^{ -2\sqrt{2\log_2(N)}}.
\end{equation}
\noindent This is based on a clever geometric construction due to Behrend from 1946 \cite{behrend}. Since \cite{behrend}, the only improvement has been to the value of the constant `$C$' (the original argument gave $C=-1/4$, but a refinement of Elkin 
 \cite{elkin} (see also \cite{greenwolf}) obtained $C= 1/4$). 
 
 In this paper, we shall get a quasipolynomial improvement to \eqref{roth bound}.
 \begin{thm}\label{main improvement}
    For $N\ge 1$, we have\[r_3(N) \gg N 2^{-(c+o(1))2 \sqrt{2\log_2(N)}},\]where
    {\tiny \[ c:= \sqrt{\log_{2}\left(\sqrt{\frac{32}{9}}\right)} \approx 0.9567. \] }
    
\end{thm}

\begin{rmk}
    The constant $9/32$ is not best possible. At the very least, one can obtain $7/24$ by using the set `$T$' from \cite[Definition~3.4]{EPS} (as we discuss at the start of Section~\ref{obtaining}). However, it seems likely that neither constant is optimal. Thus we chose to use a slightly simpler construction to focus on the core ideas. 
\end{rmk}
\subsection{Additional discussion}

It is useful to recall that Behrend-type arguments give rise to a more general parametrized range of bounds.
\begin{prp}
    Given any integers $N,D\ge 1$, we have
    \[r_3(N) \gg \sqrt{D} N2^{-D} N^{-2/D}.\]
\end{prp}
\noindent Optimizing the above by taking $D = \lceil \sqrt{2\log_2(N)}\rceil $ (so that $N\le 2^{D^2/2}$ and thus $N^{-2/D} \ge 2^{-D}$) yields the aforementioned \eqref{roth bound}.

Inspired by very recent work \cite{EPS}, we shall establish a new range of bounds.
\begin{prp}\label{main bound}
    Fix any $\ep >0$. Given integers $N,D \ge 1$, with $D$ even, we have
    \[r_3(N) \gg_\ep \frac{1}{D^2} N (\sqrt{9/32-\ep})^{D} N^{-2/D}.\]
\end{prp}

\noindent Noting that $\sqrt{9/32}>0.53>1/2$, Proposition~\ref{main bound} does noticeably better assuming that $D$ is not too big.

Whence, taking\footnote{Technically, we should take $D = 2\lceil \frac{1}{2}\sqrt{2\log_2(N)}\rceil$, so that $D$ is even, but this hardly matters.} $D= \lceil\sqrt{2 \log_2(N)}\rceil$, we obtain
\[r_3(N)\gg \log_2^{-1}(N) N (2/0.53)^{-\sqrt{2\log_2(N)}}\]which already improves the bound \eqref{roth bound} for all large $N$ since $(2/0.53)<3.78<4$. 

More carefully picking $D = \lceil \sqrt{2 \log_{\sqrt{32/9}}(N)}\rceil $ (so that $N^{-2/D} \ge (\sqrt{32/9})^{-D}$) gives
\[r_3(N) \gg N (\sqrt{9/32}-o(1))^{2D} = N(2+o(1))^{-c 2\sqrt{2\log_2(N)}},\]with $c= \sqrt{\log_2(\sqrt{32/9})}$, recovering Theorem~\ref{main improvement}.

\subsection{Outline of ideas} We start by recalling the work of Elsholtz, Proske, and Sauermann \cite{EPS} which we will build upon. 

In \cite{EPS}, the authors constructed a subset $S$ of the 2-dimesional torus $\T^2 := (\R/\Z)^2$ with density $\mu(S) = 7/24$ which had certain additive properties (cf. Proposition~\ref{qualitative}). They use this set $S$ as a building block, along with ideas from the classic Salem-Spencer construction \cite{salemspencer}, to create denser subsets of $\F_p^n$ that lack arithmetic progressions of length $3$ (using standard techniques, it easy to construct sets of size $\gg_{p,\ep} (p/2-\ep)^n$, while they manage to get $\gg_{p,\ep} (\sqrt{7/24}p-\ep)^n$).

The reason they obtain this improvement is because their 2-dimensional  ``building block'' $S$ has density $7/24>(1/2)^2$. In the standard arguments, one instead starts with a 1-dimensional building block of density $1/2$. By taking product sets of $S$ instead of the standard 1-dimensional building block, they get ``useful'' $D$-dimensional subsets which are denser (we now shrink by a $\sqrt{7/24}$-fraction per dimension, instead of $1/2$). This is simply more efficient and leads to better bounds. Recently, such types of improvements have been starting to become more frequent (see, e.g., \cite{green,EKL,HPSSS}), making delightful progress on many other additive problems.

Now, the main contribution of this paper will be to take an efficient 2-dimensional building block (like the set $S$ considered by \cite{EPS}), and get it to work in the integer setting. In principle, there are several challenges in doing this, we mention two:
\begin{itemize}
    \item There are significant differences between the group structure of $\F_p^n$ and $\Z$; \cite{EPS} raises the specific concern that the arguments of both \cite{salemspencer} and \cite{behrend} must take into account potential carrying that may occur.  
    \item Generally speaking, Salem-Spencer-type properties are a bit more flexible than those of Behrend. However, in the integer setting, Salem-Spencer arguments have a much poorer quantitative performance. So we must somehow show that the relevant building block is in fact compatible with a Behrend-type argument. 
\end{itemize}
\noindent The first bullet turns out to not be a serious issue; by using a variant of Behrend’s construction introduced by Green and Wolf \cite{greenwolf}, one can work in the torus and not worry about carries. We also note that if $N = p_1\cdots p_D$ is a square-free number, then $\Z/N\Z \cong (\Z/ p_1\Z)\times \dots \times (\Z/p_D\Z)$, which also allows one to get addition to work coordinate-wise, side-stepping carrying-related complications.

The second bullet is more worrying. Currently, in the literature, there are several constructions and tricks over finite fields using Salem-Spencer-type arguments that have not been replicated over the integers. A notable example is \cite{EKL}, which showed how to construct subsets of $\F_p^n$ of size $\gg_p (0.99 p )^n$ without arithmetic progressions of length 1000 but did not lead to any improvement for the analogous problem in the integer setting (despite the building blocks in \cite{EKL} being much more efficient than the set $S$ from \cite{EPS}). The author has previously convinced himself that the argument of \cite{EKL} seems ``fundamentally untranslatable'' to the integer setting, although in forthcoming work \cite{hunter} we will obtain comparably strong improvements for that setting via different techniques.

So while there is no reason a priori to be able to convert Salem-Spencer-type arguments into Behrend-type arguments, we manage to do this here. The first step is to show our building block \textit{almost} enjoys the additive properties used in standard Behrend constructions (modulo certain exceptions and technical caveats). This can be extracted from intermediate elementary lemmata of \cite{EPS}, although one must be more careful and precise with certain details. From there we construct a pseudonorm which still manages to capitalize on these weaker additive properties, crucially relying on the specific nature of our technical caveats. Lastly, we show how to run a modified Behrend argument using these inputs, obtaining Proposition~\ref{main bound}.

\begin{ack}
    We thank Zach Chase for his support during the writing of this note. We also thank Daniel Carter, Ben Green, and Benny Sudakov for some advice and feedback on our writing.
\end{ack}
\section{Preliminaries}\label{prelim}

\subsection{General notation}

We use asymptotic Vinogradov notation. Thus for functions $f(n),g(n)$, we say $f\gg g$ if there is some absolute constant $c> 0$ so that $f(n) \ge cg(n)$. We also write $f(n) = o(1)$ to denote a quantity which tends to $0$ as $n\to \infty$.

Given an integer $N\ge 1$, we define $[N]:= \{1,\dots, N\}$.

We define a $3$-AP to be an additive set $P$ that can be written as $P = \{x,y,z\}$ where $2y = x+z$.

\subsection{Torus notation}
We define $\T^D := (\R/\Z)^D$. 

Let $\pi:\R^D\to \T^D; x\mapsto x+\Z^D$ be the standard projection map. As an abuse of notation, let $\pi^{-1}(\theta)$ be the unique $x\in [0,1)^D$ so that $\pi(x) = \theta$.

Let $\{\cdot \}:[0,1)\to [0,1/2)$ be the map\footnote{Note that our definition of $\{\cdot\}$ differs from the standard definition.} \[\{x\} =\begin{cases}
    x& \text{for }x\in [0,1/2)\\
    x-1/2 &\text{otherwise}.
\end{cases}\]We extend the above definition coordinate-wise (so that $\{(x_1,\dots,x_D)\} = (\{x_1\},\dots,\{x_D\})$).

\subsection{Extra notation}

Since we will be working with 2-dimensional ``building blocks'', it will be convenient to have some specialized notation defined on $\T^2$.

We define the sum-map $\psi:\T^2 \to [0,2); \theta \mapsto \pi^{-1}(\theta)_1 + \pi^{-1}(\theta)_2$. We extend this coordinate-wise, i.e. for $\theta=(\theta^{(1)},\dots,\theta^{(D_0)})\in (\T^2)^{D_0}$, we define $\psi(\theta) := \bigl(\psi(\theta^{(1)}),\dots, \psi(\theta^{(D_0)})\bigr)$.

Lastly, we define the projection $p:\T^2\to [0,1)$ by $p(\theta) = \pi^{-1}(\theta)_1$. We extend this coordinate-wise in the same manner as $\psi$. 

\section{Main Construction}

\subsection{Setup}

We recall a construction from \cite{EPS} which motivates our work.
\begin{prp}\label{qualitative}
    There exists $S\subset \T^2$ with $\mu(S) = 7/24$ so that: 

    If $x,y,z\in \pi^{-1}(S)$ solve $\pi(x+z) = \pi(2y)$, then:
    \begin{itemize}
        \item $(y_1+y_2) \ge \frac{(x_1+x_2)+(z_1+z_2)}{2}$;
        \item if $x_1+x_2 = z_1+z_2$, then $\{y_1\} \ge \{x_1\}+\{z_1\}$.
    \end{itemize}
\end{prp}
\begin{rmk}
    Bullet 1 is proved in \cite[Fact 3.10]{EPS} and  Bullet 2 is proved in \cite[Fact 3.11]{EPS}.
\end{rmk}

In the next section, we shall consider truncations of (a simplified version of) their construction, to obtain quantitative analogues to Proposition~\ref{qualitative}. Precisely, we will establish the following.

\begin{prp}\label{quantitative}
    Fix any $\ep>0$. There exists $S_\ep \subset  \T^2$ with $\mu(S_\ep)\ge 9/32-\ep$ so that:

    Given any $3$-AP $\{\theta,\theta+\alpha,\theta+2\alpha\}\subset S_\ep$, we either have: 
    \[ 2\psi(\theta+\alpha)\ge \psi(\theta)+\psi(\theta+2\alpha) +1/2;\]
    or otherwise, 
    writing $d:= \frac{\pi^{-1}(\theta+2\alpha)-\pi^{-1}(\theta)}{2}$, so that
    \[\psi(\theta+2\alpha) = \psi(\theta) + 2(d_1+d_2),\]
     we have
    \begin{equation}\label{freiman 1}
        \psi(\theta+\alpha)= \psi(\theta)+(d_1+d_2).
    \end{equation}Furthermore, assuming $|d_1+d_2|\le \frac{\ep}{1000}$, we also have
    \begin{equation}\label{freiman 2}
        (1-\{p(\theta)\})^2 + (1-\{p(\theta+2\alpha)\})^2 - 2(1-\{p(\theta+\alpha)\})^2 \ge  d_1^2.
    \end{equation}
\end{prp}
\noindent The final two conditions will tell us we morally have ``Freiman homomorphism''-esque properties (in the jargon of additive combinatorics). However, since we are working with a subset of $\T^2$ with density $>1/4$, a simple averaging argument considering the subgroup $\pi(\{0,1/2\}^2)$ tells us that we can't have an \textit{actual} Freiman homomorphism. Thus it is natural and somewhat expected that the above statement is slightly messy and technical.

\subsection{Finishing things off}

Let $\vec{1}$ denote the all 1's vector in $\R^{D_0}$.

To exploit Proposition~\ref{quantitative}, we define three weights. Namely take
\[w_1(\theta) := \|\psi(\theta)\|_1 = \sum_{i=1}^{D_0} |\psi(\theta^{(i)})|, \]
\[w_{2,\ep}(\theta) := (10^{10}\ep^{-2} )\|\psi(\theta)\|_2^2 ,\]
and
\[w_3(\theta) := \|\vec{1}-p(\theta) \|_2^2 .\]

\vspace{1.5mm}

Roughly speaking, the first weight will make us win if Eq.~\ref{freiman 1} fails to hold. Otherwise, the second weight (with its large constant) will let us win if Eq.~\ref{freiman 2} fails to hold. Lastly, when both hold, considering $w_{2,\ep}(\theta)+w_3(\theta)$ shall get the job done.

Now, by considering product sets, we can establish the below: 
\begin{prp}\label{product set}
    Fix any $\ep >0$. For $D_0\ge 1$ there exists $S_{\ep,D_0}\subset (\T^2)^{D_0}$ with $\mu(S_{\ep,D})> (9/32-\ep)^{D_0}$ so that the following holds.

    \vspace{1.5mm}

    Given a $3$-AP $\{\theta,\theta+\alpha,\theta+2\alpha\}\subset S_{\ep,D_0}$, either
    \[2w_1(\theta+\alpha) \ge w_1(\theta)+w_1(\theta+2\alpha)  +1/2,\]or otherwise writing $d := \frac{\pi^{-1}(\theta+2\alpha) - \pi^{-1}(\theta)}{2} \in ([-1/2,1/2]^2)^{D_0}$, we have that 
    \vspace{2mm}
    \[\hspace{-30mm} (w_{2,\ep}(\theta)+w_3(\theta)) + (w_{2,\ep}(\theta+2\alpha)+w_3(\theta+2\alpha))\] \[\hspace{15mm} \ge 2(w_{2,\ep}(\theta+\alpha)+w_3(\theta+\alpha) )+ \sum_{i=1}^{D_0} |d^{(i)}_1+d^{(i)}_2|^2 +(d_1^{(i)})^2 .\]
\end{prp}
\begin{proof}
    Let $S_{\ep}$ be the set from Proposition~\ref{quantitative}. We take $S_{\ep,D_0}:= S_\ep^{D_0}$ (thus we get $\mu(S_{\ep,D_0}) = \mu(S_\ep)^{D_0} \ge (9/32-\ep)^{D_0}$). 

    Note that the weights sum coordinate-wise.

    We first handle $w_1$. Let $I_1$ be the set of indices $i$ with $w_1(\theta^{(i)})+w_1((\theta+2\alpha)^{(i)})\neq 2w_1((\theta+\alpha)^{(i)})$. We have that
    \[2w_1(\theta+\alpha) -(w_1(\theta)+w_1(\theta+2\alpha))\ge |I_1|/2\](since for $i\in I_1$, the $i$-th coordinate contributes $\ge 1/2$ to the RHS, and for $i\not \in I_1$, the contribution is zero). The first condition of our criteria fails to hold only when $I_1 = \emptyset$. 

    So we now assume $I_1 = \emptyset$. Write $d: = \frac{\pi^{-1}(\theta+2\alpha)-\pi^{-1}(\theta)}{2} $. By assumption, Eq.~\ref{freiman 1} holds for all $i\in [D_0]$, whence the parallelogram law gives
    \[w_{2,\ep}(\theta) +w_{2,\ep}(\theta+2\alpha) - 2 w_{2,\ep}(\theta+\alpha) =(10^{10}\ep^{-2})\sum_{i=1}^{D_0} |d^{(i)}_1+d^{(i)}_2|^2 . \]
    
    Now let $I_2 $ be the set of indices $i$ with $|d_1^{(i)}+d_2^{(i)}|\ge \frac{\ep}{1000}$. Thus, one gets \[w_{2,\ep}(\theta) +w_{2,\ep}(\theta+2\alpha) - 2w_{2,\ep}(\theta+\alpha) \ge \sum_{i=1}^{D_0} |d^{(i)}_1+d^{(i)}_2|^2 + 100|I_2|.\]At the same time (invoking Eq.~\ref{freiman 2} for the coordinates not in $I_2$, and noting $w_3(\theta)\in [0,1] $ for $\theta\in \T^2$ otherwise), we have that 
    \[ w_3(\theta) +w_3(\theta+2\alpha) - 2w_3(\theta+\alpha) \ge \sum_{i\not \in I_2} (d_1^{(i)})^2 - 2|I_2| \ge \sum_{i=1}^{D_0} (d_1^{(i)})^2 - 3|I_2|.\]

    Summing the last two equations gives the claim.
\end{proof}

We are nearly done. We require a small observation.
\begin{lem}\label{d options}
    Fix $\alpha \in (\T^2)^{D_0}$.  For each $\theta\in (\T^2)^{D_0}$, there is a choice of $\xi\in (\{-1/2,0,1/2\}^2)^{D_0}$ so that \[\frac{\pi^{-1}(\theta+\alpha)-\pi^{-1}(\theta)}{2} = \pi^{-1}(\alpha) + \xi.\]
    
\end{lem}
\begin{proof}
        We have that $\pi(\pi^{-1}(\theta+\alpha)-\pi^{-1}(\theta) ) = \alpha$, thus $\pi^{-1}(\theta+\alpha)-\pi^{-1}(\theta)-\pi^{-1}(\alpha)\in (\Z^2)^{D_0}$. Meanwhile, $\pi^{-1}(\theta+\alpha),\pi^{-1}(\theta),\pi^{-1}(\alpha)\in ([0,1)^2)^{D_0}$, thus this difference lies in $(\{-1,0\}^2)^{D_0}\subset (\{-1,0,1\}^2)^{D_0}$.  
\end{proof}

\begin{proof}[Proof of Proposition~\ref{main bound}]
    Since $D$ is even, we may write $D = 2D_0$ for some $D_0\ge 1$. 

    We now consider the set $S_{\ep,D_0}$ from Proposition~\ref{product set}.

    Set $\delta := (1/400)N^{-2/D}$. Let $B_0$ be the set of $d\in (\R^2)^{D_0}$ so that \[\sum_{i=1}^{D_0} |d^{(i)}_1+d^{(i)}_2|^2 +(d_1^{(i)})^2 < \delta.\] It is not hard to see that $B_0 \subset ([-2\sqrt{\delta},2\sqrt{\delta}]^2)^{D_0}$, thus  $\mu(B_0)\le (4\sqrt{\delta})^{D}$. 

    \vspace{2mm}
    
    Now let $B$ be the set of $\alpha \in (\T^2)^{D_0}$ so that for some choice of $\xi \in (\{-1/2,0,1/2\}^2)^{D_0}$, we have that $d:=\pi^{-1}(\alpha)+\xi\in B_0$. We get that $$\mu(B)\le 3^D\mu(B_0)\le (12\sqrt{\delta})^{D} = (12/20)^{D} \frac{1}{N}.$$

    \vspace{2mm}
    
    So, randomly picking $\theta_0\sim (\T^2)^{D_0}$, we have that $$\E[|B\cap \{\theta_0,\dots,N\cdot \theta_0\}|] <(12/20)^{D}<1$$ (since for $n=1,\dots,N$, $n\cdot \theta_0$ will be uniformly distributed over $(\T^2)^{D_0}$, and thus will lie in $B$ with probability $\mu(B)$). Thus, we can fix some outcome of $\theta_0$ where the aforementioned intersection is empty. 

    \vspace{2mm}
    
    By Proposition~\ref{product set} (recalling Lemma~\ref{d options} and the definition of $B$), it follows that for any $3$-AP $P = \{\theta,\theta+\alpha,\theta+2\alpha \}\subset S_{\ep,D_0}$ with common difference $\alpha \in \{\theta_0,\dots, N\cdot \theta_0\}$, that either $w_1(P)$ is not contained in an interval of length $1/4$, or $(w_{2,\ep}+w_3)(P)$ is not contained in an interval of length $\delta/2$.

    \vspace{2mm}

    Since $w_1((\T^2)^{D_0}) \subset [0,D]$ and $(w_{2,\ep}+w_3)((\T^2)^{D_0}) \subset [0, (10^{10}\ep^{-2}+1)D]$, we may apply pigeonhole principle to fix values $r_1,r_2$ so that \begin{align*} \mu\Bigl(S_{\ep,D_0} \cap w_1^{-1}([r_1,r_1+1/4)) \cap (w_{2,\ep}+w_3)^{-1}([r_2,r_2+\delta/2))\Bigr)  &\ge \mu(S_{\ep,D_0})\frac{\delta}{10^{11}\ep^{-2} D^2} \\ &\gg_\ep D^{-2} N^{-2/D} (9/32-\ep)^{D_0}.\end{align*} Write $\AA_0 : = S_{\ep,D_0} \cap w_1^{-1}([r_1,r_1+1/4)) \cap (w_{2,\ep}+w_3)^{-1}([r_2,r_2+\delta/2))$. This set shall play the role of our `annulus' in the Behrend-type construction. 
 
    \vspace{2mm}

    Finally, picking $\mu\sim (\T^2)^{D_0}$ randomly, we have that \[\E[\#(n\in [N]: \mu+n\cdot \theta_0 \in \AA_0)] = N\mu(\AA_0)\] (since for each $n\in [N]$ and fixed $\theta_0$, $\mu+n\cdot \theta_0$ is uniformly distributed over $(\T^2)^{D_0}$, over the randomness of $\mu$). Fix an outcome of $\mu$ where the expectation is obtained. Defining $A := \{n\in [N]: \mu+n\cdot \theta_0 \in \AA_0\}$ gives a $3$-AP-free set (indeed, if there was some $3$-AP $P= \{n_0,n_0+t,n_0+2t\}\subset A$ with common difference $t>0$,  we'd have that $\AA_0$ contains a $3$-AP with common difference $\alpha = t\cdot \theta_0\in \{\theta_0,\dots, N\cdot \theta_0\}$, contradiction). Whence $r_3(N)\ge \mu(\AA_0)N$, completing the proof.
\end{proof}

\section{Obtaining the building block}\label{obtaining}

We start by recalling the construction from \cite{EPS}. Write
\[T_1 := \left\{(a,b)\in [0,1/2)\times [1/2,1) : \frac{7}{12}\le a+b\le \frac{4}{3}\right\} \cup \left\{(a,b)\in [0,1/2)^2:  \frac{5}{6}<a+b\right\},\]
\[T_2 := \left\{(a,b) \in [1/2,1)\times [0,1/2) : \frac{7}{12}\le a+b<\frac{5}{6}\text{ and } 2a+b <\frac{3}{2}\right\}.\]
To prove Proposition~\ref{qualitative}, \cite{EPS} considered $T:= T_1\cup T_2, S:= \pi(T)$. 

We note that defining $S_\ep^* := S\setminus \psi^{-1}([5/6-\ep,5/6))$ would give a set of measure $\ge 7/24-2\ep$ that satisfies the conditions of Proposition~\ref{quantitative} (by following the proofs of \cite[Facts 3.10 and 3.11]{EPS}). However, to streamline our argument we shall work with a simpler construction.

Namely, throughout this section, we consider the following.
\begin{defn}\label{the truncation}
    Let
\[U_1 := \left\{(a,b)\in [0,1/2)\times [1/2,1)\cup [1/2,1)\times [0,1/2):  a+b< \frac{3}{4}\right\},\]
\[U_2 := \left\{(a,b)\in [0,1/2)\times [0,1) :\frac{3}{4}+\ep< a+b<\frac{5}{4}\right\};\]we set $U:= U_1\cup U_2$ and take $S_\ep = \pi(U)$.
\end{defn} \noindent It is easy to see that $$\mu(U_1) = \frac{1}{2}(1/4)^2+\frac{1}{2}(1/4)^2 = 1/16$$ while $$\mu(U_2) = \frac{1}{2} \left((3/4-\ep)^2-(1/4-\ep)^2-(1/4)^2\right)  \ge   \frac{1}{2}((9-2)/16-2\ep )\ge 7/32-\ep.$$ Whence $\mu(S_\ep)\ge (9/32)-\ep$, as desired.

\vspace{3mm}

It remains to show that $S_\ep$ has the desired additive properties required by Proposition~\ref{quantitative}. \hide{The intuition for why things should work is that: if we just worried $U' :=U\cap [0,1/2)\times [0,1)$, then $\psi(\cdot),\{p(\cdot)\}$ would both Freiman isomorphisms when restricted to this set (this is not hard to check for $\psi$, and trivial/standard for $\{p\}$). Meanwhile, $U^*:= U\setminus U' = U\cap [1/2,1)\times [0,1/2)$ is just a tiny set, which is tucked away in the corner, so you'd expect it not to `interact' with anything else.\zh{remove?}}

\vspace{3mm}
We start by giving two lemmas which we will use to establish Eq.~\ref{freiman 2}.

\begin{lem}\label{bracket law}
    Let $x,y,z \in [0,1]$ satisfy $2y= x+z$. Suppose that
    \[(1-\{x\})^2+(1-\{z\})^2 -2(1-\{y\})^2  < 2\left(\frac{z-x}{2}\right)^2 .\]Then $\min(x,z)<1/2\le y$.
    \begin{proof}

        Let $d := \frac{z-x}{2}$. By the parallelogram law, we have that $(u-d)^2+(u+d)^2 -2u^2 = 2d^2$ for all real $u\in \R$.

        Thus, (taking $u = 1-y$ and $u= 1-(y-1/2)$ respectively) one gets 
        \[(1-x)^2 + (1-z)^2 - 2(1-y)^2  = 2 d^2 = (1- (x-1/2))^2 +(1-(z-1/2))^2-2(1-(y-1/2))^2.\]WLOG assume that $x<y<z$. Now, if $z<1/2$, then we are done by the LHS above; similarly, if $x\ge 1/2$, then we are done by the RHS above. Meanwhile, if $y<1/2$ and $z\ge 1/2$, then $(1-\{z\})^2 \ge (1-z)^2$ (while $1-\{x\} = 1-x$ and $1-\{y\}=1-y$), so we are again done by the LHS.

        The final outcome which could affect $\{x\},\{y\},\{z\}$ is that $x<1/2\le y$, but here we have nothing to prove. This concludes the result.
    \end{proof}
\end{lem}

\begin{lem}\label{U1 good}
    Let $x ,z \in U_1$, and let $ y^* = \frac{x+z}{2}$ be their midpoint. Then \[(1-\{x_1\})^2+(1-\{z_1\})^2 -2(1-\{y_1^*\})^2  \ge 2\left(\frac{(z-x)_1}{2}\right)^2 .\]
    \begin{proof}
        WLOG, assume that $x_1<z_1$. If $z_1<1/2$ or $x_1\ge 1/2$, then we are immediately done by Lemma~\ref{bracket law}. So, we may assume that $x_1<1/2\le z_1 $.
        
        As $x,z\in U_1$, this implies that $x_1< 3/4-x_2 \le 1/4$ (recalling that $U_1\cap [0,1/2)^2 = \emptyset$). Meanwhile, we also have $z_1<3/4$.
        
        Thus, $x_1+z_2<1$, meaning $y_1^*<1/2$. So Lemma~\ref{bracket law} applies here too, giving the result.
    \end{proof}
\end{lem}

\begin{prp}\label{proof 2}
    Let $\theta_x,\theta_z\in S_{\ep}$ and $\theta_y\in \T^2$ satisfy $2\theta_y=\theta_x+\theta_z$. Then, writing $d:= \frac{\pi^{-1}(\theta_z)-\pi^{-1}(\theta_x)}{2}$, we have
    \[(1-\{p(\theta_x)\})^2 + (1-\{p(\theta_z)\})^2 - 2(1-\{p(\theta_y)\})^2 \ge 2 d_1^2\] assuming $|d_1+d_2|\le \ep/1000$.
    \begin{proof}
        Write $x:= \pi^{-1}(\theta_x),y:= \pi^{-1}(\theta_y), z := \pi^{-1}(\theta_z)$ and note that $x,z\in U$ by assumption. Also write $y^* := x+d = \frac{x+z}{2}$ to denote the midpoint of $x,z$. Note that $\pi(2y^*) = \pi(x)+\pi(x+2d) = \theta_x+\theta_z = \pi(2y)$, thus we must have $2y^*-2y \in \Z^2$. Since $y^*,y\in [0,1)^2$, we further get that $y^*-y\in \{-1/2,0,1/2\}^2$. 

        Hence, because $\{\cdot\}$ is appropriately periodic, it follows that $\{y^*\} = \{y\}$, and thus $ \{p(\theta_y)\} = \{y_1\} = \{y^*_1\}  $. So we are left to show that 
        \[(1-\{x_1\})^2 + (1-\{z_1\})^2 - 2(1-\{y^*_1\})^2 \ge 2 d_1^2\]assuming $|d_1+d_2|<\frac{\ep}{1000}$. Note that we must either have $\{x,z\}\subset U_1$ or $\{x,z\}\subset U_2$, since otherwise we'd have\[2|d_1+d_2| \ge \inf_{v\in U_2}(v_1+v_2)-\sup_{w\in U_1}(w_1+w_2) = (3/4+\ep)-(3/4) >2(\ep/1000).\] 

        Now, if $\{x,z\} \subset U_1$, then we are immediately done by Lemma~\ref{U1 good}. Meanwhile, if $\{x,z\}\subset U_2$, then $\{x_1,z_1\} \subset [0,1/2)$. Consequently, we get that $y_1^* \in [0,1/2)$ as well, so we are similarly done, by Lemma~\ref{bracket law}. 
    \end{proof}
\end{prp}

\begin{lem}\label{one sided rounding}
    Consider $x,z \in U^2$, and write $y^* := \frac{x+z}{2}$ to denote their midpoint. We have that $y^* \neq u +\xi $ for all $u\in U, \xi \in \{(1/2,1/2),(1/2,0),(0,1/2)\}$.
    \begin{proof}
        Write $\Xi:= \{(1/2,1/2),(1/2,0),(0,1/2)\}$.
    
        \textit{Case 1: $u\in U\cap [0,1/2)^2$}. Then we have $(u_1+u_2) >(3/4)+\ep$. Meanwhile, \[(y^*_1 +y^*_2) = \frac{(x_1+x_2)+(z_1+z_2)}{2} < \sup_{v\in U}(v_1+v_2) = (5/4)\]which is less than $5/4+\ep\le (u_1+u_2)+(\xi_1+\xi_2)$ for each choice of $\xi\in \Xi$. 

        \textit{Case 2:  $u\in U\setminus [0,1/2)^2$}. Noting that $y^* \in [0,1)^2$, we must simply rule out that $y^* \in [1/2,1)^2$ (since for $\xi \in \Xi$ and $u\in [0,1)^2\setminus [0,1/2)^2$, $u+\xi\in [0,1)^2$ happens if and only if $u+\xi \in [1/2,1)^2$).

        So now we shall prove that $y^*\not \in [1/2,1)^2$. Indeed, if $\{x,z\}\subset U_2 \subset [0,1/2)\times [0,1)$, then $y^*_1 <1/2$; otherwise, if (say) $x\in U_1$, then \[(y_1^*+y_2^*) < \frac{\sup_{v\in U_1}(v_1+v_2)+\sup_{w\in U_2}(w_1+w_2)}{2} = \frac{1}{2}((3/4)+(5/4)) = 1.\]In either case, $y^*$ exhibits something which cannot be possible if $y^*\in [1/2,1)^2$, so we are done.
    \end{proof}
\end{lem}

\begin{prp}\label{proof 1}
    Let $\theta_x,\theta_y,\theta_z\in S_\ep$ satisfy $2\theta_y =\theta_x+\theta_z$. Then, writing $d := \frac{\pi^{-1}(\theta_z)-\pi^{-1}(\theta_x)}{2}$, we have
    \[2\psi(\theta_y) \ge \psi(\theta_x) +\psi(\theta_z) +1/2\]or
    \[\psi(\theta_y) = \psi(\theta_x)+(d_1+d_2).\]
    \begin{proof}
        Write $x:= \pi^{-1}(\theta_x),y:= \pi^{-1}(\theta_y), z := \pi^{-1}(\theta_z)$ (and note that $\{x,y,z\}\subset U$ by assumption). Also write $y^* := x+d = \frac{x+z}{2}$ to denote the midpoint of $x,z$.

        As noted before (cf. Proposition~\ref{proof 2}), we have $y^*-y \in \{-1/2,0,1/2\}^2$. Write $\xi := y-y^*$; by Lemma~\ref{one sided rounding} we have that $\xi \not\in \{(-1/2,-1/2),(-1/2,0),(0,-1/2)\}$, whence $\xi_1+\xi_2\in \{0,1/2,1\}$. Now, we have
        \[ \psi(\theta_y) = (y_1+y_2) = (y_1^*+y_2^*) +(\xi_1+\xi_2).\]Meanwhile, we clearly have 
        \[(y_1^*+y_2^*) = \psi(\theta_x) + (d_1+d_2)\quad \text{and} \quad 2(y_1^*+y_2^*) = \psi(\theta_x)+\psi(\theta_z)\]by definition of $y^*$. 

        Recalling $\xi_1+\xi_2\in \{0,1/2,1\}$, we are immediately done (if $\xi_1 +\xi_2=0$, then $\psi(\theta_y) = (y_1^*+y_2^*)$ and the latter outcome holds; otherwise the former outcome holds). 
    \end{proof}
\end{prp}

\begin{proof}[Proof of Proposition~\ref{quantitative}]We take $S_\ep$ to be the set from Definition~\ref{the truncation}. The desired properties hold by combining Propositions~\ref{proof 2} and \ref{proof 1}.
\end{proof}

\section{Conclusion}

We now briefly discuss the limits of this method and some related problems.

Given an additive set $X$ and an integer $k\ge 3$, we will write $r_k(X)$ to denote the maximum cardinality of a subset $A\subset X$ which lacks non-trivial $k$-term arithmetic progressions (sets $P$ of the form $\{x_0,x_0+d,\dots,x_0+(i-1)d\}$ with $d\neq 0$).

\subsection{In finite fields}

We note that the arguments here can recover the bounds from \cite{EPS}, but with Behrend-type little order terms. Namely, using the construction $S_\ep^*$ (rather than $S_\ep$) alluded to in Section~\ref{obtaining}, one can show:
\begin{thm}
    Fix any prime $p\ge 3$. We have that \[r_3(\F_p^n) \gg_p \frac{1}{n^2} \lceil (7/24) p^2\rceil^{n/2} \gg_p  \frac{1}{n^2}(\sqrt{7/24} p)^n .\]
\end{thm}
\noindent However, this is only a very minor improvement, since \cite{EPS} proves the above only with a worse power of `$\frac{1}{n}$' (the exponent is roughly quadratic in $p$).

More interestingly, the fact that this argument for $r_3(N)$ also yields bounds for $r_3(\F_p^n)$ suggests that there is a natural barrier for how far these ideas can be na\"ively pushed. First, recall that by the celebrated work of Ellenberg-Gijswijt \cite{EG}, there exists an absolute constant $c>0$ so that $r_3(\F_p^n) \ll_p ((1-c)p)^n$ for all finite fields. 

Already to use a building block with “efficiency $>1/2$” (where the efficiency of a set $S \subset \T^t$ is $\mu(S)^{1/t}$), one has to distinguish between the image of $[N]$ embedded into $\T^D$ and $\pi(\{0,1/2\}^D$) (the image of $\F_2^D$ inside the torus). Indeed, there should be a reason why we do not prove the impossible\footnote{Recall that we wish to avoid containing \textit{any} 3-APs besides singletons. A more common non-degeneracy condition is to only worry about 3-APs with three distinct points, but this glosses over some issues that are relevant to the integer setting.} bound $r_3(\F_2^D)>1$. In the current paper, our reason comes from considering the fractional map $\{\cdot\}$, which is used to morally annihilate $\pi(\{0,1/2\}^D
)$ (thus our argument does not attempt to say anything about $3$-APs with common difference $\alpha \in \pi(\{0,1/2\}^D)$).

So similarly, if one were to achieve efficiency beyond $>1-c$, they would need a way to distinguish from $\pi(\{0,1/p,…,(p-1)/p\}^D)$ for any fixed prime $p$. This seems to suggest that one cannot “fix” a specific building block, but must instead do things in some parametrized fashion (or otherwise introduce significantly different ideas). Hence, we claim there is a natural barrier for proving $r_3(N) \gg N 2^{-\eta 2\sqrt{2\log_2(N)}}$ for some $\eta<\sqrt{\log_2(\frac{1}{1-c})}$.

\subsection{Longer progressions}

Next, one might wonder about $r_k(N):= r_k([N])$. Here the best-known lower bounds are of the form 
\[r_k(N) \gg \log^C(N) N\exp(-c_k \log^{p_k}(N))\]for certain explicit constants $c_k,p_k$ (cf. \cite{obryant} which refines the results of \cite{rankin} via the framework of \cite{greenwolf}). 

For experts in this area, we note that a routine adjustment to the methods \cite{obryant}, where we now use our building block $S_\ep$ in the final stage of the bootstrapping procedure, allows one to show that
\[r_k(N) \gg N\exp(-(1-\eta)c_k \log^{p_k}(N))\]for some absolute constant $\eta>0$ and all $k\ge 3$. However, without an improved building block, such ideas cannot replace `$1-\eta$' by (say) `$1/10$' (even when $k$ is very large).

We mention that in forthcoming work, the author will give a more conceptual argument to obtain these sorts of improvements for all sufficiently large $k$ \cite{hunter}. There we will construct `increasingly efficient' building blocks as $k$ gets sufficiently large, thus the bounds obtained in \cite{hunter} (for all large $k$) will outperform whatever can be obtained by just na\"ively applying the work from the present paper.

\end{document}